\documentclass[12pt,leqno,english]{smfart}

\usepackage{dsfont,eucal,mathptmx,txfonts}

\usepackage[text={7in,9in},centering]{geometry}

\usepackage[dvipsnames]{xcolor}   
\usepackage{xparse}
\usepackage{xr-hyper}
\definecolor{brightmaroon}{rgb}{0.76, 0.13, 0.28}
\usepackage[linktocpage=true,colorlinks=true,hyperindex,citecolor=teal,linkcolor=blue]{hyperref}   

\renewcommand{\labelenumi}{\upshape(\arabic{enumi})}

\usepackage{mathtools}
\usepackage{multirow}
\usepackage{graphicx}	
\usepackage{float}
\usepackage[all]{xy}
\restylefloat{table} 
\usepackage{etoolbox}

\newcommand{\inv}{^{\raisebox{.2ex}{$\scriptscriptstyle-1$}}}   
\newtheorem{theorem}{Theorem}[section] 
\newtheorem{proposition}[theorem]{Proposition}
\newtheorem{lemma}[theorem]{Lemma}
\newtheorem{corollary}[theorem]{Corollary}
\theoremstyle{definition}

\newtheorem{example}[theorem]{Example} 
\newtheorem{remark}[theorem]{Remark}
\renewcommand{\labelenumi}{\upshape\arabic{enumi}.}
  
\title{$k$-spaces of non-domain-valued geometric points} 

\author{Amartya Goswami}

\address{[1] Department of Mathematics and Applied Mathematics, University of Johannesburg, P.O. Box 524, Auckland Park, 2006, South Africa\\ 
[2] National Institute for Theoretical and Computational Sciences (NITheCS), South Africa} 
\email{agoswami@uj.ac.za}

\begin{abstract}
The aim of this paper is to study the topological properties of algebraic sets with zero divisors. We impose a subbasic topology on the set of proper ideals of a $k$-algebra and this new ``$k$-space'' becomes a generalization of the corresponding Zariski space. We prove that a $k$-space is $T_{\scriptscriptstyle 0}$, quasi-compact, spectral, and connected. Moreover, we study continuous maps between such $k$-spaces. We conclude with a question about construction of a sheaf of $k$-spaces similar to affine schemes.
\end{abstract}  



\begin{document}   
\maketitle 

\section{Introduction}

In the introduction of \cite{G71}, \textsc{Grothendieck} described the process of getting the spectrum of prime ideals (also called geometric points) starting from a system of polynomial equations. In brief, it is as follows.

Suppose $k$ is a commutative ring with identity. Let $P_I=k[(x_i)_{i\in I}]$ be a ring of polynomials in the indeterminates $x_i$ with coefficients in $k,$ and  $I$ be an index set (not necessarily finite). 
Let $S=\{p_j\}_{j\in J}$ be a system of polynomials of $P_I,$ where the index set $J$ is also not necessarily finite. An element $a=(a_i)_{i\in I}$ of a $k$-algebra $A$ is called a \emph{solution} of the system $S$ if $p_j(a)=0$ for all $j \in J.$

If $\mathds{A}\mathrm{lg}_k$ and $\mathds{S}\mathrm{ets}$ respectively denote the categories of $k$-algebras and sets, then a functor $$\mathcal{V}_S\colon \mathds{A}\mathrm{lg}_k \to \mathds{S}\mathrm{ets}$$ represents the solutions of some system $S$ of polynomial equations with coefficients in $k$ if and only if $\mathcal{V}_S$ is representable, \textit{i.e.}, $\mathcal{V}_S$ is isomorphic to $\mathcal{R}_A$ for some object $A$ in  $\mathds{A}\mathrm{lg}_k.$ Conversely, for every object $A$ in $\mathds{A}\mathrm{lg}_k$  the representable functor $\mathcal{R}_A$ is isomorphic to some $\mathcal{V}_S.$
The functor $\mathcal{R}_A$ is called the \emph{affine algebraic space over} $k$ represented by $A.$ The category $\mathds{A}\mathrm{ff}_k$ of affine algebraic spaces has representable functors $\mathcal{R}_A$ ($A$ is an object in $\mathds{A}\mathrm{lg}_k$) as objects and morphisms are defined as natural transformations, \textit{i.e.}, they are the induced maps $$\mu(f)\colon \mathrm{Hom}_{\mathds{A}\mathrm{lg}_k} (A, B) \to \mathrm{Hom}_{\mathds{A}\mathrm{lg}_k} (A', B)$$ obtained from morphisms $f\colon A'\to A$ in $\mathds{A}\mathrm{lg}_k.$
An $A'$-\emph{valued point} is a $k$-algebra homomorphism $f\colon A\to A',$ \textit{i.e.}, $f$  is an element of the set $\mathrm{Hom}_{\mathds{A}\mathrm{lg}_k} (A, A').$ If we restrict $A'$ to be an object of the full subcategory $\mathds{F}\mathrm{ield}_k$ of $\mathds{A}\mathrm{lg}_k$ then the elements of $\mathrm{Hom}_{\mathds{A}\mathrm{lg}_k} (A, A')$ are called \emph{geometric points}.

We define an equivalence relation between geometric points as follows.
We say two geometric points $f'\colon A\to A'$ and $f''\colon A\to A''$ are equivalent if there exists a third geometric point $f\colon A\to A_1$ and $k$-algebra morphisms $g'\colon A'\to A_1$ and $g''\colon A''\to A_1$ such that 
 \begin{equation}
f=g''\circ f''=g'\circ f'
\label{com1}
\end{equation}	\textit{i.e.}, 
the following diagram commutes:
\[
\xymatrix{ A\ar[r]^{f'}\ar[d]_{f''}\ar[dr]^{f} & A'\ar[d]^{g'}\\
A'' \ar[r]^{g''}& A_1. 
}
\]

Since $g'$ and $g''$ are monomorphisms, we observe that the  condition (\ref{com1}) is equivalent to $\mathrm{ker}f'=\mathrm{ker}f''.$ Therefore, the equivalence classes of the above relation are in bijection with the prime ideals of $A.$
Now there is a bijection between geometric points and prime ideals of a $k$-algebra $A.$
The \emph{loci} of a $k$-algebra $A$ is the equivalence classes of geometric points.

The \emph{spectrum} of $A$ (denoted by $\mathrm{Spec}\,A$)  is defined as the set of prime ideals of $A,$ \textit{i.e.}, 
$X=\mathrm{Spec}\,A= \{\mathfrak{p}\mid \mathfrak{p}\;\text{is a prime ideal of}\;A \}.$
For $S=\{p_j\}_{j\in J}$ be a system of polynomials of $P_I,$ let $\mathcal{V}(S)$ be the subset of $\mathrm{Spec}\,A$  defined by
$$\mathcal{V}(S)=\{\text{set of loci of}\;u\in \mathcal{R}_A(A')\mid f_i(u)=0, \forall_{i\in I}\, f_i\in S \},$$
where $f_i(u)$ is defined by $f_i(u)=u(f_i).$
From the above definition of $\mathcal{V}(S),$ we immediately see
$$\mathcal{V}(S) = \{\mathfrak{p}\in X \mid S\subseteq \mathfrak{p}\}.$$
From the above, we observe that in order to obtain the $\mathrm{Spec}\,A$, we worked with the full subcategory $\mathds{F}\text{ield}_k$ of $\mathds{A}\text{lg}_k$. 

If we remove this `restriction' on the $k$-algebra, what we obtain is a spectrum $\mathrm{Idl}\,A$ of  all ideals (instead of prime ideals) of $A$. This will allow us to study polynomial equations having solutions in any $k$-algebra (not necessarily a field). Since $A\notin \mathrm{Spec}\,A,$ we also consider the set $\mathrm{Spi}\,A$ of all proper ideals of $A$ as our `generalized' spectrum on which we endow a topology and call it a $k$-space. Our choice of the notation $\mathrm{Spi}\,A$ is to have an `alignment' with the notation $\mathrm{Spm}\,A$ of maximal ideals of $A$ as in \textsc{Grothendieck} \cite{G71}. A $k$-space is a generalization of a Zariski space (\textit{i.e.}, $\mathrm{Spec}\,A$ endowed with a Zariski topology). The purpose of this paper is to study topological properties of $k$-spaces and simultaneously compare them with Zariski spaces.

\section{$k$-spaces}

To construct a $k$-space, we use two maps defined in Proposition \ref{kalg}. Similar maps also appear when we take values of polynomials over an integral domain (to impose a Zariski topology on a $\mathrm{Spec}\,A$).  Before we discuss properties of these maps, let us see an example in our context.

\begin{example}
We consider a $k$-algebra with zero divisors and its algebraic sets. The polynomials listed in the Table\,\ref{tabA} are of minimal degrees.
	
\begin{table}[H]
\begin{center}
\begin{tabular}{|c|c|c|c|}
\hline
\footnotesize\textbf{Subsets of} $\mathbb{Z}_4$ & \footnotesize\textbf{Polynomials}& \footnotesize\textbf{Algebraic sets of} $\mathbb{Z}_4$&\footnotesize\textbf{Polynomials}\\
\hline
\footnotesize				$\emptyset$ &\footnotesize$\mathbb{Z}_4\!\!\setminus\!\! \{0\}$ &\footnotesize$\emptyset$ & \footnotesize$1$\\
\hline 
\footnotesize$\{0\}$&\footnotesize$\{ax \}$ &\footnotesize$\{0\}$ &\footnotesize$x$\\
\hline
\footnotesize$\{1\}$&\footnotesize$\{ax+b\mid a+b=0  \}$&\footnotesize$\{1\}$ &\footnotesize$x+3$\\
\hline 
\footnotesize$\{2\}$&\footnotesize$\{ax+b\mid 2a+b=0  \}   $&\footnotesize$\{2\}$ &\footnotesize$(x+2)$\\
\hline 
\footnotesize$\{3\}$&\footnotesize$\{ax+b\mid 3a+b=0\}$ &\footnotesize$\{3\}$ &\footnotesize$x+1$\\
\hline 
\footnotesize$\{0, 1\}$&\footnotesize$\{ax^2+bx+c\mid a+b=0, c=0\}$&\footnotesize$\{0,1\}$ &\footnotesize$x(x+3)$ \\
\hline
\footnotesize$\{0, 2\}$&\footnotesize$\{ax^2+bx+c\mid 2b=0, c=0\}$ &\footnotesize$\{0,2\}$ &\footnotesize$x(x+2)$\\
\hline
\footnotesize$\{0, 3\}$& \footnotesize$\{ax^2+bx+c\mid a+3b=0, c=0\}$ &\footnotesize$\{0,3\}$ &\footnotesize$x(x+1)$\\
\hline
\footnotesize$\{1, 2\}$&\footnotesize$\{ax^2+bx+c\mid a-b=0, 2a+c=0\}$&\footnotesize$\{1,2\}$ &\footnotesize$(x+3)(x+2)$\\
\hline 
\footnotesize$\{1, 3\}$&\footnotesize$\{ax^2+bx+c\mid 2b=0, 2a+2c=0\}$ &\footnotesize$\{1,3\}$ &\footnotesize$(x+3)(x+1)$\\
\hline
\footnotesize$\{2, 3\}$&\footnotesize$\{ax^2+bx+c\mid a+b=0, 2b+c=0\}$ &\footnotesize$\{2,3\}$ &\footnotesize$(x+2)(x+1)$\\
\hline
\footnotesize$\{0, 1, 2\}$&\footnotesize$\{ax^2+bx+c\mid a+b=0, 2b=0\}$ &\footnotesize$\{0,1,2\}$ &\footnotesize$x(x+3)(x+2)$\\\hline
\footnotesize$\{0, 1, 3\}$ &\footnotesize$\{ax^2+bx+c\mid a+b=0, 2b=0\}$ &\footnotesize$\{0,1,3\}$ &\footnotesize$x(x+3)(x+1)$\\\hline
\footnotesize$\{0, 2, 3 \}$&\footnotesize$\{ax^2+bx+c\mid a+b=0, 2b=0\}$ &\footnotesize $\{0,2,3\}$&\footnotesize$x(x+2)(x+1)$
\\\hline
\footnotesize$\{1,2,3\}$&\footnotesize$\{ax^2+bx+c\mid a+b=0, 2b=0\}$ &\footnotesize$\{1,2,3\}$ &\footnotesize$(x+3)(x+2)(x+1)$
\\\hline  
\footnotesize$\mathbb{Z}_4$& \footnotesize$\{ax+b\mid a=0, b=0\}$ &\footnotesize$\mathbb{Z}_4$ &\footnotesize$x(x+3)(x+2)(x+1)$
\\\hline
\end{tabular}
\end{center}
\caption{Algebraic sets of $\mathbb{Z}_4$}
\label{tabA}
\end{table}
\end{example}

\begin{proposition}\label{kalg}
Define the maps
$\def\objectstyle{\scriptstyle}
\def\labelstyle{\scriptstyle}
\xymatrix@C=1cm{
\mathrm{Spi}\,A\ar@<.4ex>[r]^{\mathcal{I}}
&\mathfrak{P}(A)\ar@<.4ex>[l]^{\mathcal{V}}
}$
as follows: $$\mathcal{V}(\{x\})=\{\mathfrak{a}\in \mathrm{Spi}\,A\mid x\in\mathfrak{a}\},\quad \mathcal{I}(S)=\cap\{\mathfrak{s}\mid\mathfrak{s}\in S\},$$ with $S\subseteq \mathrm{Spi}\,A$. Here $\mathfrak{P}(A)$ denotes the power set of $A$. Then $\mathcal{V}$ and $\mathcal{I}$ satisfy the following properties.
	
\begin{enumerate}
\item $\mathcal{V}(S)=\mathcal{V}(\left<S\right>),$
where $\left<S\right>$ is the ideal of $A$ generated by the subset $S.$  

\item The map $\mathcal{V}$ is order reversing and surjective.
		
\item \label{vavra} If  $\mathfrak{a}$ is a non-radical ideal of $A$,  then $\mathcal{V}(\mathfrak{a})\subsetneq \mathcal{V}(\sqrt{\mathfrak{a}})$ if and only if $A$ has non-zero zero divisors.
		
\item\label{ncl} For any two ideals $\mathfrak{a},$ $\mathfrak{b}$ of $A$, we have $$\mathcal{V}(\mathfrak{a})\cup \mathcal{V}(\mathfrak{b}) \subseteq \mathcal{V}(\mathfrak{a}\cap \mathfrak{b})\subseteq \mathcal{V}(\mathfrak{ab}). $$

\item For a family of sets $\{ \mathcal{V}(\mathfrak{a}_{\alpha}) \}_{\alpha \in \Gamma},$ we have $\bigcap_{\alpha \in \Gamma}\mathcal{V}(\mathfrak{a}_{\alpha})=\mathcal{V}\left( \sum_{\alpha \in \Gamma}\mathfrak{a}_{\alpha} \right).$

\item\label{vaa0} $\mathcal{V}(\mathfrak{a})=\mathrm{Spi}\,A$ if and only if $\mathfrak{a}=\mathfrak{o},$ where $\mathfrak{o}$ is the zero ideal of $A.$  If $\mathcal{V}(\mathfrak{a})=\emptyset$, then $\mathfrak{a}=A.$

\item For any two ideals $\mathfrak{a},$ $\mathfrak{b}$ of $A$ and $\mathfrak{b}\subseteq \sqrt{\mathfrak{a}}$ implies $\mathcal{V}(\sqrt{\mathfrak{a}})\subseteq \mathcal{V}(\mathfrak{b}).$ 

\item The map $\mathcal{I}$ is order reversing and surjective.

\item $\mathcal{I}(\emptyset)=A$ and  $\mathcal{I}\left( \bigcup_{\lambda \in \Lambda}T_{\lambda} \right)=\bigcap_{\lambda \in \Lambda} \mathcal{I}\left(T_{\lambda}\right). $ 

\item\label{ivvi} If $T$ is a subset of $\mathrm{Spi}\,A$ and  $\mathfrak{a}$ is an ideal of $A,$ then $\mathcal{IV}(\mathfrak{a})\supseteq  \mathfrak{a},$ and $\mathcal{VI}(T)=T.$ 

\item\label{chchk} the collections $\mathcal{C}_{\mathcal{V}}=\{ \mathcal{V}(\mathfrak{a})\mid \mathfrak{a}\in \mathrm{Idl}(A)\}$ and  $\mathcal{C}_{\mathcal{VI}}=\{ \mathcal{VI}(S)\mid S\in \mathcal{P}(\mathrm{Spi}\,A)\}$ of sets are identical, where $\mathrm{Idl}\,A$ denotes the poset (under inclusion) of all ideals of $A$.   
\end{enumerate}
\end{proposition}

\begin{remark}
Notice that for $\mathrm{Spec}\,A$ and for any ideal $\mathfrak{a}$ of $A$, we always have equalities in (\ref{vavra}) and (\ref{ncl}).  Note that for $\mathrm{Spec}\,A$, we always have: $\mathcal{IV}(\mathfrak{a})=\sqrt{\mathfrak{a}},$ the radical of $\mathfrak{a}$ (\textit{cf.} Proposition \ref{kalg} (\ref{ivvi}))
\end{remark}

\subsection{$k$-topologies}
In case of $\mathrm{Spec}\,A$, the sets $\{\mathcal{V}(\mathfrak{a})\}_{\mathfrak{a}\in \mathrm{Idl}\,A}$ are closed under finite unions and we obtain the usual Zariski topology on $\mathrm{Spec}\,A$. But that closure property fails to hold for $\mathrm{Spi}\,A$ (see Theorem \ref{kalg}(\ref{ncl})).
However, as a sub-base,  $\mathcal{C}_{\mathcal{V}}$ or equivalently by $\mathcal{C}_{\mathcal{VI}}$ (see Proposition \ref{kalg} (\ref{chchk})) induces a unique  topology on  $\mathrm{Spi}\,A$, which 
we call   the \emph{$k$-topology}. We denote the corresponding  topological space by $(\mathrm{Spi}\,A, \mathcal{C}_{\mathcal{V}}),$ and in short, call it a \emph{$k$-space}. With the abuse of notation we shall also denote the space by $\mathrm{Spi}\,A.$ A $k$-topology coincides with the Zariski topology whenever we restrict $\mathrm{Spi}\,A$ to $\mathrm{Spec}\,A$. Note that a study of a similar topology on various classes of ideals of a ring has been done in \textsc{Dube} \& \textsc{Goswami} \cite{DG22}.

It is well-known that a \emph{Zariski space} is quasi-compact. The same holds for a $k$-space. In the proof we shall use the
Alexander Subbase Theorem.

\begin{proposition}\label{comp} 
A $k$-space is quasi-compact. 
\end{proposition} 

\begin{proof}   
Let  $\{K_{ \alpha}\}_{\alpha \in \Lambda}$ be a family of subbasic closed sets of an $k$-space $\mathrm{Spi}\,A$   such that $\bigcap_{\alpha\in \Lambda}K_{ \alpha}=\emptyset.$ Let $\{\mathfrak{s}_{ \alpha}\}_{\alpha \in \Lambda}$ be a family of ideals of $A$ such  that $\forall \alpha \in \Lambda,$  $K_{ \alpha}=\mathcal{V}(\mathfrak{s}_{ \alpha}).$  Since $$\bigcap_{\alpha \in \Lambda}\mathcal{V}(\mathfrak{s}_{ \alpha})=\mathcal{V}\left(\sum_{\alpha \in \Lambda}\mathfrak{s}_{ \alpha}\right),$$ we get  $\mathcal{V}\left(\sum_{\alpha \in \Lambda}\mathfrak{s}_{ \alpha}\right)=\emptyset,$ and that by Proposition \ref{kalg} (\ref{vaa0})  implies $ \sum_{\alpha \in \Lambda}\mathfrak{s}_{ \alpha}=A.$ Then, in particular, we obtain $1=\sum_{\alpha_i\in \Lambda}s_{ \alpha_i},$ where $s_{ \alpha_i}\in \mathfrak{s}_{\alpha_i}$ and $s_{ \alpha_i}\neq 0$ for $i=1, \ldots, n$. This implies    $A=\sum_{  i \, =1}^{ n}\mathfrak{s}_{\alpha_i}.$ Therefore,   $\bigcap_{ i\,=1}^{ n}K_{ \alpha_i}=\emptyset,$ and hence by Alexander subbase theorem, $\mathrm{Spi}\,A$ is quasi-compact.  
\end{proof}  

Since $\mathcal{V}(\mathfrak{a})\neq \mathcal{V}(\mathfrak{a}')$ for any two distinct elements $\mathfrak{a}$ and $\mathfrak{a}'$ of $\mathrm{Idl}\,A$, we immediately have

\begin{proposition}
Every $k$-space is   $T_{\scriptscriptstyle 0}.$ 
\label{ct0t1}  
\end{proposition}
 
It is known that $\{\mathcal{V}(\mathfrak{p})\mid \mathfrak{p}\in \mathrm{Spec}\,A\}$ are exactly the irreducible closed subsets of a Zariski space. For a $k$-space, the situation is more intriguing. 

\begin{theorem}\label{irrc}
Every non-empty subbasic closed subset of a $k$-space  is irreducible.
\end{theorem} 

\begin{proof} 
Since for every non-empty subbasic closed subset $\mathcal{V}(\mathfrak{a})$ of a $k$-space $\mathrm{Spi}\,A$, the ideal $\mathfrak{a}$ is also in $\mathrm{Spi}\,A,$ it is sufficient to show that $\mathcal{V}(\mathfrak{a})=\mathcal{C}\!\ell(\mathfrak{a})$ for every $\mathfrak{a}\in\mathrm{Spi}\,A$. Observe that $\mathcal{C}\!\ell(\mathfrak{a})$ is the smallest closed set containing $\mathfrak{a}$ and  $\mathcal{V}(\mathfrak{a})$ is a closed set such that $\mathfrak{a}\in\mathrm{Spi}\,A.$Therefore,  $\mathcal{C}\!\ell(\mathfrak{a})\subseteq \mathcal{V}(\mathfrak{a})$.To obtain the reverse inclusion, first consider the case: $\mathcal{C}\!\ell(\mathfrak{a})= \mathrm{Spi}\,A$. Since
$$
\mathrm{Spi}\,A=\mathcal{C}\!\ell(\mathfrak{a})\subseteq \mathcal{V}(\mathfrak{a})\subseteq \mathrm{Spi}\,A,
$$
we obtain $\mathcal{V}(\mathfrak{a})=\mathcal{C}\!\ell(\mathfrak{a})$. Now, let $\mathcal{C}\!\ell(\mathfrak{a})\neq \mathrm{Spi}\,A$. For $\mathcal{C}\!\ell(\mathfrak{a})$, there exists an  index set,  $\Omega$, such that for each $\alpha\in\Omega$, there is a positive integer $n_{\alpha}$ and $\mathfrak{a}_{\alpha 1},\dots, \mathfrak{a}_{\alpha n_\alpha}\in\mathrm{Idl}\,A$ such that 
$$
\mathcal{C}\!\ell(\mathfrak{a})={\bigcap_{\alpha\in\Omega}}\left({\bigcup_{ i\,=1}^{ n_\alpha}}\mathcal{V}(\mathfrak{a}_{\alpha i})\right).
$$
Since by hypothesis,
$\mathcal{C}\!\ell(\mathfrak{a})\neq \mathrm{Spi}\,A$, without loss of generality, assume that ${\bigcup_{ i\,=1}^{ n_\alpha}}\mathcal{V}(\mathfrak{a}_{\alpha i})\neq \emptyset$, for each $\alpha$. Therefore, $\mathfrak{a}\in   {\bigcup_{ i\,=1}^{ n_\alpha}}\mathcal{V}(\mathfrak{a}_{\alpha i})$, for each $\alpha$, and from that we have $$\mathcal{V}(\mathfrak{a})\subseteq {\bigcup_{ i=1}^{ n_\alpha}}\mathcal{V}(\mathfrak{a}_{\alpha i}),$$ \textit{i.e.}, $\mathcal{V}(\mathfrak{a})\subseteq \mathcal{C}\!\ell(\mathfrak{a})$, and this completes the proof.
\end{proof}     

A Zariski space $\mathrm{Spec}\,A$ is connected if and only if the $k$-algebra $A$ does not have any non-trivial idempotent elements. For a $k$-space the situation is much simpler.

\begin{theorem}\label{conis}
Every $k$-space $\mathrm{Spi}\,A$ is connected.
\end{theorem}  

\begin{proof}
Since by Proposition \ref{kalg} (\ref{vaa0}), $\mathrm{Spi}\,A=\mathcal{V}(\mathfrak{o})$ and since irreducibility implies connectedness, the desired claim immediately follows from Theorem \ref{irrc}. 
\end{proof}

It is known that every Noetherian 
space can be represented as a finite union of non-empty irreducible closed subsets.
For a $k$-space $\mathrm{Spi}\,A,$ this representation is always possible irrespective of $A$ being Noetherian and hence $\mathrm{Spi}\,A$ being Noetherian. This follows from the fact that $\mathcal{V}(\mathfrak{o})$ is irreducible in $\mathrm{Spi}\,A$. 

Next, we wish to prove that every non-empty irreducible closed subset of a $k$-space has a unique generic point. To this end, notice that if $K$ is an irreducible closed subset of  a topological space $X$ and $\mathcal{S}$ is a closed subbase of $X$, then it is known (see \textsc{Harris} \cite[\textsection 7.2]{H73}) that $K$ is the intersection of members of $\mathcal{S}.$ For a $k$-space we get more. In other words, the converse of Theorem \ref{irrc} is also true. 

\begin{lemma}\label{ircs} 
If $K$ is a non-empty irreducible closed subset of a $k$-space $\mathrm{Spi}\,A$, then $K=\mathcal{V}(\mathfrak{a})$ for some $\mathfrak{a}\in \mathrm{Spi}\,A.$ 
\end{lemma}

\begin{proposition}\label{sob}  
Every $k$-space is sober.
\end{proposition}  

\begin{proof}  
It follows from Lemma \ref{ircs} that every non-empty irreducible closed subset of $\mathrm{Spi}\,A$ is of the form $\mathcal{V}(\mathfrak{a})$, where $\mathfrak{a}\in \mathrm{Spi}\,A.$ Let $\mathcal{V}(\mathfrak{a})$ be a non-empty irreducible closed subset of $\mathrm{Spi}\,A.$ Since $\mathfrak{a}\in\mathcal{V}(\mathfrak{a})$, we have $\mathcal{C}\!\ell(\mathfrak{a})\subseteq \mathcal{V}(\mathfrak{a}).$ Therefore, to show $\mathcal{V}(\mathfrak{a})$ has a generic point, it is now sufficient to show that $\mathcal{C}\!\ell(\mathfrak{a})\supseteq \mathcal{V}(\mathfrak{a}).$ Since $\mathcal{C}_{\mathcal{V}}$ is a closed subbase of $\mathrm{Spi}\,A$, the required containment follows from Lemma  \ref{irrc}. Moreover, by Proposition \ref{ct0t1}, every  $k$-space is $T_{\scriptscriptstyle  0}$. So, we have the uniqueness of a  generic point. 
\end{proof}  

According to \textsc{Hochster} \cite{H69}, a topological space is called \emph{spectral} if it is  quasi-compact, sober,  admitting a basis of quasi-compact open subspaces that is closed under finite intersections. It has also been shown in \cite{H69} that a Zariski space is spectral. We wish to show that a $k$-space is also spectral and our proof is constructible topology-independent and avoids the  checking of the existence of a basis of quasi-compact open subspaces that is closed under finite intersections. The key to our proof is the following 

\begin{lemma}\label{cso}
A quasi-compact, sober, open subspace of a spectral space is spectral. 
\end{lemma}

\begin{proof}
Suppose $S$ is a quasi-compact, sober, open subspace of a spectral space $X$. Since $S$ is quasi-compact and sober, it is sufficient to prove that the set $\mathcal{O}_{\scriptscriptstyle S}$ of compact open subsets of $S$ forms a basis of a topology that is closed under finite intersections. It is obvious that a subset $T$ of $S$ is open in $S$ if and only if $T$ is open in $X$, and hence a subset $T$ of $S$ belongs to $\mathcal{O}_{\scriptscriptstyle S}$ if and only if $T$ belongs to $\mathcal{O}_{\scriptscriptstyle X}.$ Now 
using these facts, we argue as follows.

Let $U$ be an open subset of $S$. Since $U$ is also open in $X$, we have $U=\cup\, \mathcal{U},$ for some subset $\mathcal{U}$ of $\mathcal{O}_{\scriptscriptstyle X}.$ But each element of $\mathcal{U}$ being a subset of $U$ is a subset of $S$, and it belongs to $\mathcal{O}_{\scriptscriptstyle S}.$ Therefore, every open subset of $S$ can be presented as a union of compact open subsets of $S$. Now it remains to prove that $\mathcal{O}_{\scriptscriptstyle S}$ is closed under finite intersections, but this immediately follows from the fact that $\mathcal{O}_{\scriptscriptstyle X}$ is closed under finite intersections. 
\end{proof}

\begin{theorem}
Every $k$-space is spectral.
\end{theorem}

\begin{proof}
It is well known (see \textsc{Priestley} \cite[Theorem 4.2]{P94}) that the set $\mathrm{Idl}\,A$ endowed with a $k$-topology is spectral. Now, if we extend the domain of $\mathcal{V}$ to $\mathrm{Idl}\,A$, then it is easy to see that with some routine changes of notation in the proof, Theorem \ref{irrc} still holds. Moreover, we have  $\{A\}=\mathcal{V}(A)=\mathcal{C}\!\ell(A),$ and therefore $\mathrm{Idl}\,A {\setminus}\mathrm{Spi}\,A$ is closed, and that implies $\mathrm{Spi}\,A$ is open. The desired claim now follows from Lemma \ref{cso}, Proposition \ref{comp}, and Proposition \ref{sob}. 
\end{proof}

Once we have $k$-spaces, it is natural to consider the continuous maps between such spaces. Using subbasic-closed-set formulation of continuity, we obtain the following properties. 

\begin{proposition}\label{conmap}
Let $\phi\colon A\to A'$ be a $k$-algebra homomorphism  and $\mathfrak{b}\in\mathrm{Spi}\,A'.$ Then
\begin{enumerate}
		
\item \label{contxr} the map $\phi^*\colon  \mathrm{Spi}\,A'\to \mathrm{Spi}\,A$ defined by  $\phi^*(\mathfrak{b})=f\inv(\mathfrak{b})$ is    continuous$;$
		
\item \label{shcs} if $\phi$ is  surjective, then the $k$-space $\mathrm{Spi}\,A'$ is homeomorphic to the closed subspace $\mathcal{V}(\mathrm{ker}\,\phi)$ of the $k$-space $\mathrm{Spi}\,A;$
		
\item \label{imde} the image  $\phi^*(\mathrm{Spi}\,A')$ is dense in $\mathrm{Spi}\,A$ if and only if $$\mathrm{ker}\,\phi\subseteq \bigcap_{\mathfrak{s}\in \mathrm{Spi}\,A}\mathfrak{s};$$
		
\item \label{loca} if $A_{S}$ is the localization of a $k$-algebra $A$ at a multiplicative closed subset $S$, then there is a closed, continuous, and injective map from the $k$-space  $\mathrm{Spi}(R_{ S})$ to the $k$-space $$(\mathrm{Spi}\,A)_S:=\{ \mathfrak{s}\in \mathrm{Spi}\,A\mid \mathfrak{s}\cap S=\emptyset\}.$$  
\end{enumerate}
\end{proposition}

\begin{proof}      
To show (\ref{contxr}), let $\mathcal{V}(\mathfrak{a})$ be a   subbasic closed set of the ideal  space $\mathrm{Spi}\,A.$ Observe  that $$(\phi^*)\inv(\mathcal{V}(\mathfrak{a})) =\{\mathfrak{b}\in \mathrm{Spi}\,A'\mid \phi(\mathfrak{a})\subseteq \mathfrak{b}\}=\mathcal{V}(\langle \phi(\mathfrak{a})\rangle),$$ and hence the map $\phi^*$  continuous.     
For the homeomorphism in (\ref{shcs}), observe that     
$\mathrm{ker}\,\phi\subseteq \phi\inv(\mathfrak{b}),$ in other words, $\phi^*(\mathfrak{b})\in \mathcal{V}(\mathrm{ker}\,\phi)$. This implies that $\mathrm{im}\,\phi^*=\mathcal{V}(\mathrm{ker}\,\phi).$  
Since for all $\mathfrak{b}\in \mathrm{Spi}\,A',$ $$\phi(\phi^*(\mathfrak{b}))=\mathfrak{b}\cap \mathrm{im}\,\phi=\mathfrak{b},$$ the map $\phi^*$ is injective. To show that $\phi^*$ is a closed map, first we observe that for any subbasic closed subset  $\mathcal{V}(\mathfrak{a})$ of  $\mathrm{Spi}\,A'$, we have $$\phi^*(\mathcal{V}(\mathfrak{a}))=\phi\inv\{ \mathfrak{i}'\in \mathrm{Spi}\,A'\mid \mathfrak{a}\subseteq   \mathfrak{i}'\}=\mathcal{V}(\phi\inv(\mathfrak{a})).$$   Now if $K$ is a closed subset of $\mathrm{Spi}\,A'$ and if $$K=\bigcap_{ \alpha \in \Omega} \left(\bigcup_{ i \,= 1}^{ n_{\alpha}} \mathcal{V}(\mathfrak{a}_{ i\alpha})\right),$$ then $$\phi^*(K)=\phi\inv \left(\bigcap_{ \alpha \in \Omega} \left(\bigcup_{ i = 1}^{ n_{\alpha}} \mathcal{V}(\mathfrak{a}_{ i\alpha})\right)\right)=\bigcap_{ \alpha \in \Omega} \bigcup_{ i = 1}^{ n_{\alpha}} \mathcal{V}(\phi\inv(\mathfrak{a}_{ i\alpha})),$$ a closed subset of  $\mathrm{Spi}\,A.$ Since $\phi^*$ is   continuous, we have the proof.
To prove (\ref{imde}), we first show that $\mathcal{C}\!\ell(\phi^*(\mathcal{V}(\mathfrak{b})))=\mathcal{V}(\phi\inv(\mathfrak{b})),$ for all ideals $\mathfrak{b}\in R'.$ To this end, let $\mathfrak{s}\in \phi^*(\mathcal{V}(\mathfrak{b})).$ This implies $\phi(\mathfrak{s})\in \mathcal{V}(\mathfrak{b}),$ which means $\mathfrak{b}\subseteq \phi(\mathfrak{s}).$ In other words, $\mathfrak{s}\in \mathcal{V}(\phi\inv(\mathfrak{b})).$ The other inclusion follows from the fact that $\phi\inv(\mathcal{V}(\mathfrak{b}))=\mathcal{V}(\phi\inv(\mathfrak{b})).$ Since $$\mathcal{C}\!\ell(\phi^*(\mathrm{Spi}\,A'))=\mathcal{V}(\phi\inv(\mathfrak{o}))=\mathcal{V}(\mathrm{ker}\,\phi),$$ the closed subspace $\mathcal{V}(\mathrm{ker}\,\phi)$ is equal to $\mathrm{Spi}\,A$ if and only if $\mathrm{ker}\,\phi\subseteq \cap_{\mathfrak{s}\in \mathrm{Spi}\,A}\mathfrak{s}.$ 
Finally, to have (\ref{loca}), it is easy to see that the ring homomorphism $\phi \colon A\to A_{S}$ defined by $\phi(r)=r/1$ induces a map $\phi^*\colon \mathrm{Spi}\,A_S\to \mathrm{Spi}\,A$ defined by $\phi^*(\mathfrak{a})=\phi\inv(\mathfrak{a}).$ We claim that  $\phi^*(\mathfrak{a})\cap S=\emptyset.$ If not, let $s\in \phi^*(\mathfrak{a})\cap S.$ Then $$\phi(s)\in \phi(\phi\inv(\mathfrak{a})\cap S)=\phi(\phi\inv(\mathfrak{a}))\cap \phi(S)=\mathfrak{a}\cap \phi(S),$$ and hence $\phi(s)\in \mathfrak{a}.$ Since $\phi(s)$ is a unit in $A_{S},$ this implies $\mathfrak{a}=A_{ S},$ a contradiction. Therefore, $\phi^*$ is indeed a map from $\mathrm{Spi}\,A_{S}$ to $(\mathrm{Spi}\,A)_S.$ If $\phi^*(\mathfrak{a})=\phi^*(\mathfrak{b})$ for some $\mathfrak{a},$ $\mathfrak{b}\in \mathrm{Spi}\,A_{S},$ then $$\mathfrak{a}=\phi(\phi\inv(\mathfrak{a}))=\phi(\phi\inv (\mathfrak{b}))=\mathfrak{b}$$ shows that $\phi^*$ is injective.   
The map    $\phi^*\colon\mathrm{Spi}\,A_S\to\mathrm{Spi}\,A\backslash S$ is continuous follows from (\ref{contxr}). Since $\phi^*(\mathcal{V}(\mathfrak{a}))=\mathcal{V}(\phi\inv(\mathfrak{a})),$ the map $\phi^*$ is also closed. Therefore, $\phi^*$ has the desired properties.    
\end{proof}    

\begin{corollary}
The $k$-space $\mathrm{Spi}\,(A/\mathfrak{a})$ is homeomorphic to the closed subspace
$\mathcal{V}(\mathfrak{a})$ of $\mathrm{Spi}\,A$.   
\end{corollary}

\begin{remark}
From Proposition  \ref{conmap}, we get the well-known result that the Zariski  spaces $\mathrm{Spec}\,A$ and $\mathrm{Spec}(A/\sqrt{\mathfrak{o}})$ are canonically homeomorphic, and  $\phi^*(\mathrm{Spec}\,A')$ is dense in $\mathrm{Spec}\,A$ if and only if $\mathrm{ker}\,\phi\subseteq \mathcal{V}(\mathfrak{o}).$	
\end{remark}

\section{Conclusion}

The generalizations like schemes, algebraic spaces of algebraic varieties still do not answer how to do algebraic geometry when polynomial equations have solutions over a $k$-algebra which is not an integral domain. Inclusion of zero divisors immediately brings the following two problems in Grothendieck's scheme theory:

\begin{enumerate}
\item \label{nzt} the algebraic sets $\{\mathcal{V}(\mathfrak{a})\}_{\mathfrak{a}\in \mathrm{Idl}\,A}$ no longer form a Zariski topology; and
	
\item \label{lor} we do not get a sheaf of local rings.
\end{enumerate}

The alternative topology ($k$-topology) proposed here handles the problem (\ref{nzt}). But we do not know how to resolve problem (\ref{lor}). Considering $\mathrm{Spi}\,A$ is not good enough to obtain a sheaf of local rings. For, if $S=A\!\!\setminus\!\! \mathfrak{s},$ for some $\mathfrak{s}\in \mathrm{Spec}\,A$, then $S$ is a multiplicatively closed set and hence we can have a local algebra $A_S$ at $S$. But if $\mathfrak{s}\in \mathrm{Spi}\,A$, then the set $S$ is no longer necessarily multiplicatively closed. Therefore, inclusion of zero divisors in varieties might require a completely different treatment to study local properties of geometric objects.

Table \ref{tabB} summarizes major results in this paper and compared them with Zariski spaces.

\begin{table}[H]
\begin{center}
\begin{tabular}{|c|c|}
\hline
\footnotesize\textbf{EGA I} & \footnotesize\textbf{Proposed idea}\\
\hline\hline 
\footnotesize Field-valued polynomial equations &\footnotesize $k$-algebra-valued polynomial equations \\
\hline
\footnotesize The set $\mathrm{Spec}\,A$ of prime ideals of $A$ &\footnotesize The set $\mathrm{Spi}\,A$ of proper ideals of $A$\\
\hline
\footnotesize $\mathcal{V}(\mathfrak{a})=\{\mathfrak{p}\in\mathrm{Spec}\,A\mid \mathfrak{a}\subseteq \mathfrak{p}\}$ &\footnotesize $\mathcal{V}(\mathfrak{a})=\{\mathfrak{s}\in\mathrm{Spi}\,A\mid \mathfrak{a}\subseteq \mathfrak{s}\}$\\
\hline
\footnotesize $\mathcal{V}(\mathfrak{a})\cup \mathcal{V}(\mathfrak{b}) = \mathcal{V}(\mathfrak{a}\cap \mathfrak{b})= \mathcal{V}(\mathfrak{ab}). $&\footnotesize $\mathcal{V}(\mathfrak{a})\cup \mathcal{V}(\mathfrak{b}) \subseteq \mathcal{V}(\mathfrak{a}\cap \mathfrak{b})\subseteq \mathcal{V}(\mathfrak{ab}). $ \\
\hline
\footnotesize $\mathcal{V}(\mathfrak{a})=\mathcal{V}(\sqrt{\mathfrak{a}})$& \footnotesize $\mathcal{V}(\mathfrak{a})\supseteq \mathcal{V}(\sqrt{\mathfrak{a}})$\\
\hline 
\footnotesize $\mathcal{IV}(\mathfrak{a})=\sqrt{\mathfrak{a}}$ &\footnotesize $\mathcal{IV}(\mathfrak{a})\supseteq \mathfrak{a},$\\
\hline
\footnotesize Zariski topology &\footnotesize $k$-topology \\\hline
\footnotesize Compact and $T_{\scriptscriptstyle 0}$&\footnotesize Compact and $T_{\scriptscriptstyle 0}$\\
\hline
\footnotesize $\{\mathcal{V}(\mathfrak{p})\mid \mathfrak{p}\in \mathrm{Spec}\,A\}$ are irreducible  &\footnotesize Non-empty subbasic closed sets are irreducible\\
\hline
\footnotesize Sober &\footnotesize Sober\\
\hline
\footnotesize Spectral&\footnotesize Spectral\\
\hline
\footnotesize $\mathrm{Spec}\,A$ is connected iff $A$ does not &\footnotesize $\mathrm{Spi}\,A$ is always connected\\\footnotesize have any non-trivial idempotent elements& \\
\hline
\end{tabular}
\end{center}
\caption{Zariski spaces vs.\,$k$-spaces}
\label{tabB}
\end{table}

\section*{Acknowledgement}
The author wishes to extend heartfelt gratitude to the anonymous referee for their thorough review and invaluable feedback, which greatly enhanced the paper's presentation.

\end{document}